\documentclass[a4paper,12pt]{article}
\usepackage[top=3.25cm, bottom=3cm, left=3cm, right=3cm]{geometry}

\usepackage{amssymb,amsmath}
\usepackage{amsthm}
\usepackage{mathrsfs}
\usepackage{enumerate}
\usepackage[colorlinks]{hyperref}
\usepackage{tikz}
\usetikzlibrary{shapes,arrows,patterns}
\usepackage{caption}
\usepackage{subfig}
\usepackage{float}
\newtheorem{Lemma}{Lemma}[section]
\newtheorem{Proposition}[Lemma]{Proposition}
\newtheorem{Theorem}[Lemma]{Theorem}
\newtheorem{Conjecture}{Conjecture}

\theoremstyle{definition}
\newtheorem{Definition}[Lemma]{Definition}
\newtheorem{Example}[Lemma]{Example}
\theoremstyle{remark}

\newcommand{\Pic}{\mathrm{Pic}}
\newcommand{\mgnbar}{\overline{M}_{g,n}}
\newcommand{\monbar}{\overline{M}_{0,n}}
\newcommand{\BB}{\mathscr{B}}
\newcommand{\QQ}{\mathbb{Q}}
\newcommand{\QQp}{\mathbb{Q}_{\geqslant 0}}
\newcommand{\ZZ}{\mathbb{Z}}

\newcommand{\bA}{\boldsymbol{a}}
\newcommand{\bC}{\boldsymbol{c}}
\newcommand{\bW}{\boldsymbol{w}}
\newcommand{\bE}{\boldsymbol{e}}
\newcommand{\bV}{\boldsymbol{v}}
\newcommand{\bP}{\boldsymbol{p}}

\renewcommand{\geq}{\geqslant}
\renewcommand{\le}{\leqslant}
\renewcommand{\ge}{\geqslant}

\title{Towards Fulton's conjecture}
\author{Claudio Fontanari, Riccardo Ghiloni, Paolo Lella}
\date{}
\begin{document}
\maketitle

\abstract{We present an alternate proof, much quicker and more straightforward than the original one, 
of a celebrated Fulton's conjecture on the ample cone of the moduli space 
$\monbar$ of stable rational curves with $n$ marked points in the case $n=7$.}


\section{Introduction} 

A quote traditionally attributed to Lipman Bers says that \emph{God created the natural numbers and compact Riemann surfaces, while the rest of mathematics in man made}. Indeed, compact Riemann surfaces (or, equivalently, smooth projective complex algebraic curves) turn out to be a very natural and fundamental object in mathematics. Since the pioneering work of Riemann, it is known that a compact Riemann surface with $g$ holes (namely, of \emph{genus} $g$) carries a complex structure depending on $3g-3$ parameters (or \emph{moduli}). More precisely, there is a complex 
algebraic variety $M_g$ of dimension $3g-3$ parameterizing smooth curves of genus $g$. Of course $M_g$ cannot be compact, since smooth 
curves degenerate to singular ones. There is however a canonical compactification $\overline{M}_g$ of $M_g$, the 
so-called \emph{Deligne-Mumford compactification}, parameterizing only mildly singular curves, the so-called \emph{stable} curves, with at most ordinary nodes as singularities and finite automorphism group. In order to provide an efficient 
description of the codimension one boundary $\partial \overline{M}_g = \overline{M}_g \setminus M_g$, it is useful 
to introduce moduli spaces of pointed curves $\mgnbar$, parameterizing the data $(C; p_1, \ldots, p_n)$, where 
$p_1, \ldots, p_n$ are distinct smooth points on the nodal curve $C$  and there are only finitely many automorphisms 
of $C$ fixing $p_1, \ldots, p_n$. From this definition it follows that $\overline{M}_g = \overline{M}_{g,0}$ and that 
the boundary components of $\mgnbar$ are images of natural gluing maps defined either on $\overline{M}_{g-1,n+2}$ 
or on $\overline{M}_{g_1,n_1+1} \times \overline{M}_{g_2,n_2+1}$, with $g_1+g_2=g$ and 
$n_1+n_2=n$. 

The moduli space $\mgnbar$ is an irreducible complex projective variety of dimension $3g-3+n$ and the problem of 
classifying compact Riemann surfaces translates into the study of the projective geometry of $\mgnbar$. From this 
point of view, one of the basic questions to be addressed is the description of the ample cone of $\mgnbar$, 
since (suitable multiples of) ample divisors on a projective variety define all its projective embeddings. In the case 
of $\mgnbar$ there is an explicit conjectural characterization of the ample cone, usually referred to as 
\emph{Fulton's conjecture} (see \cite{GKM:02}, Conjecture (0.2)). The main result in \cite{GKM:02} is the 
so-called \emph{Bridge Theorem} (0.3), stating that Fulton's conjecture holds for $\mgnbar$ for every $g \ge 0$ 
and for every $n \ge 0$ if and only if it holds for $\monbar$ for every $n \ge 3$. This is indeed quite powerful 
and rather surprising: in order to understand the geometry of $\mgnbar$ in arbitrary genus $g \ge 0$ it is 
sufficient to address the first case $g=0$, where $\monbar$ is a smooth algebraic variety birational to a
projective space of dimension $n-3$. Unluckily, Fulton's conjecture turns out to be terribly hard even for $g=0$. 
The best results available today go back to 1996, when the case $n \le 7$ was checked in \cite{KMK:96}, 
Theorem 1.2(3), by exploiting the fact that $\overline{M}_{0,7}$ is \emph{nearly log Fano}, in the sense that 
its anticanonical divisor is effective. 

A few years later, a more combinatorial approach was proposed in \cite{GKM:02}. Namely, let $\Delta_S$ with 
$S \subset \{1, \ldots, n \}$, $2 \le \vert S \vert \le n-2$, denote the boundary component of $\monbar$ 
whose general point parameterizes the union of two rational curves $C_1 \cup C_2$ together with $n$ 
points $(p_1, \ldots, p_n)$ such that $p_i \in C_1$ if and only if $i \in S$. By definition, we have 
$\Delta_S = \Delta_{S^c}$. According to \cite{Keel:92}, (2) on p. 550, the Picard group $\Pic(\monbar)$ 
of divisors modulo linear equivalence is freely generated by the boundary divisors $\Delta_S$ modulo 
the following set of relations: 
\begin{equation}\label{eq:Vn}
\mathscr{V}_n := \left\{ \sum_{ \begin{subarray}{c} a,b \in S\\ c,d \notin S \end{subarray}} \Delta_S - 
\sum_{ \begin{subarray}{c}a,c \in S \\ b,d \notin S\end{subarray}} \Delta_S\ \Bigg\vert\begin{array}{c} a,b,c,d \in \{1,\ldots,n\}\\ a,b,c,d \text{ distinct}\end{array}\right\}. 
\end{equation}
Let $\delta_S$ denote the equivalence class of $\Delta_S$ modulo $\langle \mathscr{V}_n\rangle$. The content of \cite{GKM:02}, 
Question (0.13), is the following
\begin{Conjecture}\label{main}
Let $D = \sum_{S} a_S \delta_S$ be a divisor on $\monbar$ such that
\[
a_{I \cup J} + a_{I \cup K} + a_{I \cup L} \ge a_I + a_J + a_K + a_L
\]
for every partition $\{1, 2, \ldots, n \} = I \cup J \cup K \cup L$ into $4$ disjoint and nonempty subsets. 
Then $D = \sum_{S} b_S \delta_S$ for suitable $b_S \ge 0$.
\end{Conjecture}
More geometrically, this means that if a divisor $D$ intersects non-negatively an explicit class of 
one-dimensional subvarieties of $\monbar$ (namely, the irreducible components of the one-dimensional 
stratum of the stratification of the space of stable rational curves by topological type), then $D$ is 
linearly equivalent to an effective combination of boundary divisors. 

The authors of \cite{GKM:02} remarked that Conjecture \ref{main} implies the original Fulton's conjecture 
via an easy inductive argument and checked the case $n \le 6$ of Conjecture \ref{main} using a computer 
program. A theoretical proof for $n \le 6$ was soon provided by \cite{FarGib:03}, Theorem 2, and an 
alternate proof was presented in \cite{Fontanari:05}, by introducing a convenient basis of $\Pic(\monbar)$
 defined inductively as follows:
\begin{equation}\label{eq:baseFontanari}
\begin{split}
\BB_4 :={}& \left\{ \delta_{\{2,3\}} \right\},\\
\BB_n :={}& \BB_{n-1} \cup \big\{ \delta_B\ \big\vert\ \{n-1, n-2 \} \subseteq B \subseteq \{1, \ldots, n-1 \} \big\}\\ 
&{}\cup \big\{\delta_{B^c \setminus \{ n \}}\ \big\vert\ \delta_B \in \BB_{n-1} \setminus \BB_{n-2} \big\}.
\end{split}
\end{equation}
The case $n=7$ of Conjecture \ref{main} turned out to be much more difficult, as pointed out in \cite{GKM:02}, 
p. 277: \emph{Unfortunately, the computational complexity is enormous, and beyond our machine's capabilities 
already for $n=7$}. A tour-de-force proof for $n=7$ was finally completed in \cite{Larsen:12} and a couple of 
years later a counterexample to Conjecture \ref{main} for $n=12$ was produced in \cite{Pixton:13}. 

Here instead we go back to the case $n=7$ and we present an alternate proof, much quicker and more straightforward than the original one. 
In order to do so, we exploit once again the basis  $\BB_n$, together with different techniques borrowed from 
convex geometry.  

This research was partially supported by FIRB 2012 \lq\lq Moduli spaces and Applications\rq\rq\ and by GNSAGA of INdAM (Italy).

\section{Convex geometry interpretation}
Following the steps of Larsen \cite{Larsen:12}, we rephrase Conjecture \ref{main} in terms of polyhedral cones in a finite dimensional rational vector space. We consider the vector space generated by the boundary components of $\monbar$. More precisely, we consider the boundary components indexed by the set $\mathfrak{S}$ of subsets in $\{1,\ldots,n\}$ of cardinality at least 2 in which we pick only one subset between $S$ and $S^c$ (because $S$ and $S^c$ define the same boundary component). 
We denote by $\boldsymbol{W}_n$ the rational vector space
\begin{equation}\label{eq:ambientSpace}
\boldsymbol{W}_n := \QQ \langle \Delta_S\ \vert\ S \in \mathfrak{S} \rangle\quad\text{and}\quad  N := \dim\boldsymbol{W}_n = 2^{n-1}-n-1,
\end{equation}
by $\boldsymbol{V}_n$ the subspace generated by the set $\mathscr{V}_n$ in \eqref{eq:Vn} and by $\Pic(\monbar)_{\QQ}$ the quotient space $\boldsymbol{W}_n/\boldsymbol{V}_n = \Pic(\monbar) \otimes_{\ZZ} \QQ$. Recall that 
\[
\begin{split}
&\overline{N} := \dim \Pic(\monbar)_{\QQ} = 2^{n-1}-\binom{n}{2}-1 \text{ and }\\
& M:= \dim \boldsymbol{V}_n = \binom{n}{2}-n=\frac{n(n-3)}{2}.
\end{split}
\]
Throughout the paper, we denote a generic vector $\sum_S a_S \Delta_S \in \boldsymbol{W}_n$ by $\bA$ and $[\bA]$ stands for the corresponding element $(\ldots,a_S,\ldots) \in \QQ^N$. Furthermore, we denote by $\phi_n$ the projection map $\boldsymbol{W}_n \xrightarrow{\phi_n} \boldsymbol{W}_n/\boldsymbol{V}_n$.

\begin{Definition}
We call $\mathcal{F}_n$ the \emph{$F$-nef cone} contained in $\boldsymbol{W}_n$ defined by
\begin{equation}\label{eq:FnefCone}
\mathcal{F}_n := \bigcap_{I,J,K,L} \left\{ \bA\in\boldsymbol{W}_n\ \Big\vert\ H_{I,J,K,L}([\bA]) \geqslant 0\right\},
\end{equation}
where $H_{I,J,K,L}$ is the linear form
\[
w_{I \cup J} + w_{I \cup K} + w_{I \cup L} - w_I - w_J - w_K - w_L.
\]
\end{Definition}

The conjecture can be restated as follows.
\setcounter{Conjecture}{0}
\begin{Conjecture}[Convex geometry formulation A]
For every $\sum_{S} a_S \Delta_S$ in $\mathcal{F}_n$, there exists $\sum_S c_S \Delta_S \in \boldsymbol{V}_n$ such that $\sum_S a_S \Delta_S + \sum_S c_S \Delta_S$ is contained in the positive orthant of $\boldsymbol{W}_n$, i.e.~$a_S + c_S \geqslant 0$ for all $S \in \mathfrak{S}$.
\end{Conjecture}
\noindent Indeed, both $\sum_S a_S \Delta_S$ and $\sum_S(a_S+c_S)\Delta_S$ belong to $\phi_n^{-1}(\sum_S a_S \delta_S)$, so that $\sum_S a_S \delta_S = \sum_S(a_S+c_S)\delta_S \in \Pic(\monbar)_{\QQ}$ and $\sum_S(a_S+c_S)\delta_S$ is an effective representation of $\sum_S a_S \Delta_S$.

\smallskip

Now, we want to describe more explicitly the set of vectors that can be obtained from the positive orthant $\mathcal{O}_n$ in $\boldsymbol{W}_n$ by translation of elements in $\boldsymbol{V}_n$. 

\begin{Definition}
We denote by $\mathcal{E}_n$ the Minkowski sum between $\mathcal{O}_n$ and $\boldsymbol{V}_n$, i.e.~the set of vectors
\begin{equation}\label{eq:transEffCone}
\mathcal{E}_n := \mathcal{O}_n+\boldsymbol{V}_n = \left\{ \bE = \bP + \bV\ \big\vert\ \bP \in \mathcal{O}_n \text{ and } \bV \in \boldsymbol{V}_n \right\}.
\end{equation}
\end{Definition}

We can further restate the conjecture as follows.
\setcounter{Conjecture}{0}
\begin{Conjecture}[Convex geometry formulation B]
The F-nef cone $\mathcal{F}_n$ is contained in the cone $\mathcal{E}_n$.
\end{Conjecture}

From an effective point of view, proving the conjecture for $n \geqslant 7$ by checking the containment of $\mathcal{F}_n$ in $\mathcal{E}_n$ seems to be not feasible due to the large dimension of the vector spaces involved.
The main difficulty is given by the kind of representations of the two cones that naturally arises from the definitions. In fact,
we have a V-representation of $\mathcal{E}_n$, i.e.~we know a finite set of vectors of $\boldsymbol{W}_n$ such that all elements of $\mathcal{E}_n$ can be described as positive linear combinations. Precisely, given a basis $\{\boldsymbol{v}^{(1)},\ldots,\boldsymbol{v}^{(M)}\}$ of $\boldsymbol{V}_n$, each element of $\mathcal{E}_n$ has a decomposition with non-negative coefficients in terms of the $N+2M = 2^{n-1} + n(n-4)-1$ vectors
\[
 \QQp\langle\Delta_S\ \vert\ S \in \mathfrak{S}\rangle + \QQp\langle\boldsymbol{v}^{(1)},\ldots,\boldsymbol{v}^{(M)}\rangle + \QQp \langle -\boldsymbol{v}^{(1)},\ldots,-\boldsymbol{v}^{(M)}\rangle.
\] 

To check the containment of $\mathcal{F}_n$ in $\mathcal{E}_n$, we would need its V-representation as well. However, the cone $\mathcal{F}_n$ is described as intersection of half-spaces, i.e.~we know the H-representation. Determining a V-representation of $\mathcal{F}_n$ means to compute the extremal rays of the cone and, as already discussed in \cite{Larsen:12}, this is out of reach for $n \geqslant 7$. 

Conversely, one can try to compute the H-representation of $\mathcal{E}_n$. In principle, this approach might work better because a partial enlargement of $\mathcal{O}_n$ may suffice, namely we could try to accurately choose a vector subspace $\boldsymbol{V}' \subset \boldsymbol{V}_n$ such that
\[
\mathcal{F}_n \subset \mathcal{O}_n + \boldsymbol{V}' \subset \mathcal{O}_n + \boldsymbol{V}_n = \mathcal{E}_n.
\]
Unfortunately, computational experiments suggest that the whole subspace $\boldsymbol{V}_n$ is needed.
Indeed, in both cases $n=5$ and $n = 6$, we can explicitly compute $\mathcal{O}_n + \boldsymbol{V}_n$, but we did not succeed in finding a smaller subspace $\boldsymbol{V}'$ such that $\mathcal{F}_n \subset \mathcal{O}_n + \boldsymbol{V}'$ (see Example \ref{ex:M05} and Proposition \ref{prop:case6formB} in next section). In the case $n = 7$, we could not manage to find a subspace $\boldsymbol{V}'$ for which the H-representation of $\mathcal{O}_7 + \boldsymbol{V}'$ can be computed (in reasonable time) and such that $\mathcal{F}_7$ is contained in $\mathcal{O}_7 + \boldsymbol{V}'$. This is due to the huge growth of the number of inequalities coming out in the construction of $\mathcal{O}_7 + \boldsymbol{V}'$. To see this, we give a worst-case estimate of the number of inequalities that is exponential in the dimension of $ \boldsymbol{V}'$. The H-representation of $\mathcal{O}_n + \boldsymbol{V}'$ can be obtained from such inequalities by determining which of them describe facets of the cone. This problem is formally known as  \emph{redundancy removal problem} in the context of polyhedral computation and, in presence of large sets of data, its solution can be very time consuming.
 
Let $\{\boldsymbol{v}^{(1)},\ldots,\boldsymbol{v}^{(d)}\}$ be a basis of $\boldsymbol{V}'$. The filtration of vector spaces $\QQ\langle \boldsymbol{v}^{(1)} \rangle \subset \QQ\langle \boldsymbol{v}^{(1)}, \boldsymbol{v}^{(2)} \rangle \subset \cdots \subset \QQ\langle \boldsymbol{v}^{(1)},\ldots,\boldsymbol{v}^{(d)}\rangle$ induces the filtration of cones
\begin{equation}\label{eq:filtrationE}
\mathcal{O}_n \subset \mathcal{O}_n + \QQ\big\langle \boldsymbol{v}^{(1)} \big\rangle \subset \cdots \subset \mathcal{O}_n + \QQ\big\langle\boldsymbol{v}^{(1)},\ldots,\boldsymbol{v}^{(d)}\big\rangle.
\end{equation}
Let $\mathcal{E}_n^{(0)} = \mathcal{O}_n$ and denote by $\mathcal{E}_n^{(i)}$ the cone $\mathcal{O}_n + \QQ\langle \boldsymbol{v}^{(1)},\ldots,\boldsymbol{v}^{(i)}\rangle$. We have $\mathcal{E}_n^{(i)} = \mathcal{E}_n^{(i-1)} + \QQ\langle \boldsymbol{v}^{(i)}\rangle$. Assume to know the H-representation of $\mathcal{E}_n^{(i-1)}$:
\[
\mathcal{E}_n^{(i-1)} = \bigcap_{j=1}^{h_{i-1}} \left\{H_j^{(i-1)}\geqslant 0\right\}.
\]
The generic element $\bW \in \boldsymbol{W}_n$ is contained in $\mathcal{E}_n^{(i)}$ if, and only if, there exists $t \in \QQ$ such that $\bW + t \bV^{(i)}$ is contained in $\mathcal{E}_n^{(i-1)}$: 
\[
\bW \in \mathcal{E}_n^{(i)} \ \Leftrightarrow \ \exists\ t \text{~s.t.~} H_j^{(i-1)}([\bW]+t[\bV^{(i)}]) \geqslant 0,\ \forall\ j = 1,\ldots,h_{i-1}.
\]
Hence, we obtain the system of inequalities 
\[
\begin{cases}
H^{0}_j([\bW]) \geqslant 0,& j=1,\ldots,h^0_{i},\\
t \geqslant H^{+}_j([\bW]),& j=1,\ldots,h^+_{i},\\
t \leqslant H^{-}_j([\bW]),& j=1,\ldots,h^-_{i},
\end{cases}
\]
where $H^{0,+,-}_j$ are linear forms in the variables $w_S$ and $h^0_{i} + h^+_{i} + h^-_{i} = h_{i-1}$. The first type of inequality arises whenever the parameter $t$ does not appear in $H_j^{(i-1)}([\bW]+t[\bV^{(i)}])$, whereas the second and third type of inequality arise when $t$ appears (depending on the sign of its coefficient).  The H-representation of $\mathcal{E}_n^{(i)}$ can be deduced by the inequalities that ensures the existence of $t \in \QQ$. We need to consider the intersection of the following $h^0_{i} + h^+_{i} \cdot h^-_{i}$ half-spaces:
\begin{align*}
& H^{0}_j \geqslant 0, && j=1,\ldots,h_i^0\\ 
& H^{-}_{j_-} -  H^{+}_{j_+} \geqslant 0, && j_-=1,\ldots,h_i^-,\quad j_+=1,\ldots,h_i^+. 
\end{align*}
The worst case would happen whenever at each step we have $h_i^0 = 0$ and $h_i^+ = h_i^- = h_{i-1}/2$. Then, the number of half-spaces needed in the $H$-representation of $\mathcal{E}_n^{(d)}$ is bounded by
\[
\begin{split}
h_d^{\textsc{max}}&{} = \left(\frac{h^{\textsc{max}}_{d-1}}{2}\right)^2 = \frac{1}{2^2} \left(\frac{h^{\textsc{max}}_{d-2}}{2}\right)^{2^2} = \cdots = \frac{\left(h_0^\textsc{max}\right)^{2^d}}{2^{2+\cdots+2^d}}  = \frac{N^{2^d}}{4^{2^d-1}},
\end{split}
\]
since the positive orthant $\mathcal{O}_n = \mathcal{E}_n^{(0)}$ is obviously given by the intersection of $N$ half-spaces. This means that the number of half-spaces defining $\mathcal{E}_n$ is bounded by
\[
h_M^{\textsc{max}} = \frac{N^{2^M}}{4^{2^M-1}} = \frac{\left(2^{n-1}-n-1\right)^{2^{n(n-3)/2}}}{4^{2^{n(n-3)/2-1}}} \in O\left(2^{2^{n^2 \log n}}\right). 
\]

\begin{Example}\label{ex:M05}
Let us look at Conjecture \ref{main} in the case of the moduli space $\overline{M}_{0,5}$ viewed in terms of the second convex geometry formulation. The vector space $\boldsymbol{W}_5$ is generated by the boundary components
\[
\Delta_{1, 2},\, \Delta_{1, 3},\, \Delta_{1, 4},\, \Delta_{1, 5},\, \Delta_{2, 3},\, \Delta_{2, 4},\, \Delta_{2, 5},\, \Delta_{3, 4},\, \Delta_{3, 5},\, \Delta_{4, 5},
\]
the subspace $\boldsymbol{V}_5$ is spanned by
\begin{align*}
& \bV^{(1)} = \Delta_{1, 2}-\Delta_{1, 3}-\Delta_{2, 4}+\Delta_{3, 4},
&& \bV^{(2)} = \Delta_{1, 2}-\Delta_{1, 4}-\Delta_{2, 3}+\Delta_{3, 4},\\
& \bV^{(3)} = \Delta_{1, 2}-\Delta_{1, 3}-\Delta_{2, 5}+\Delta_{3, 5},
&& \bV^{(4)} = \Delta_{1, 2}-\Delta_{1, 5}-\Delta_{2, 3}+\Delta_{3, 5},\\
& \bV^{(5)} = \Delta_{1, 2}-\Delta_{1, 4}-\Delta_{2, 5}+\Delta_{4, 5},
\end{align*}
and the $F$-nef cone $\mathcal{F}_5$ is defined by the following 10 half-spaces:
\begin{align*}
& w_{3, 4}+w_{3, 5}+w_{4, 5} -w_{1, 2}\geqslant 0, &&
w_{2, 4}+w_{2, 5}+w_{4, 5} -w_{1, 3} \geqslant 0,\\ 
& w_{2, 3}+w_{2, 5}+w_{3, 5} -w_{1, 4} \geqslant 0,&&
w_{2, 3}+w_{2, 4}+w_{3, 4} -w_{1, 5}  \geqslant 0,\\
& w_{1, 4}+w_{1, 5}+w_{4, 5}-w_{2, 3} \geqslant 0,&&
 w_{1, 3}+w_{1, 5}+w_{3, 5} -w_{2, 4} \geqslant 0,\\
&  w_{1, 3}+w_{1, 4}+w_{3, 4}-w_{2, 5} \geqslant 0,&&
   w_{1, 2}+w_{1, 5}+w_{2, 5}-w_{3, 4}  \geqslant 0,\\
&    w_{1, 2}+w_{1, 4}+w_{2, 4}-w_{3, 5} \geqslant 0,&&
     w_{1, 2}+w_{1, 3}+w_{2, 3}-w_{4, 5} \geqslant 0.
\end{align*}
We determine explicitly the inequalities defining the second cone $\mathcal{E}_5^{(1)} = \mathcal{O}_5 + \QQ\langle\bV^{(1)}\rangle$ in the filtration \eqref{eq:filtrationE}. The generic vector $\bW$ is contained in $\mathcal{E}_5^{(1)}$ if, and only if, there exists $t\in\QQ$ such that $\bW + t \bV^{(1)}$ is contained in the positive orthant, namely
\[
\begin{cases}
w_{1,2} + t \geqslant 0\\
w_{1,3} - t \geqslant 0\\
w_{1,4} \geqslant 0\\
w_{1,5} \geqslant 0\\
w_{2,3} \geqslant 0\\
w_{2,4} - t \geqslant 0\\
w_{2,5} \geqslant 0\\
w_{3,4} + t \geqslant 0\\
w_{3,5} \geqslant 0\\
w_{4,5} \geqslant 0
\end{cases}
\Longrightarrow\quad
\begin{cases}
t \geqslant -w_{1,2}\\
t \leqslant w_{1,3}\\
w_{1,4} \geqslant 0\\
w_{1,5} \geqslant 0\\
w_{2,3} \geqslant 0\\
t \leqslant w_{2,4} \\
w_{2,5} \geqslant 0\\
 t \geqslant -w_{3,4}\\
w_{3,5} \geqslant 0\\
w_{4,5} \geqslant 0
\end{cases}
\Longrightarrow\quad
\begin{cases}
w_{1,4} \geqslant 0\\
w_{1,5} \geqslant 0\\
w_{2,3} \geqslant 0\\
w_{2,5} \geqslant 0\\
w_{3,5} \geqslant 0\\
w_{4,5} \geqslant 0\\
w_{1,2} + w_{1,3}  \geqslant 0 \\
w_{1,2} + w_{2,4} \geqslant 0 \\
w_{1,3} + w_{3,4} \geqslant 0\\
w_{2,4} + w_{3,4} \geqslant 0\\
\end{cases}.
\]
Iterating the process, we find $12$, $15$, $22$ and $37$ inequalities defining the cones $\mathcal{E}_5^{(2)}$, $\mathcal{E}_5^{(3)}$, $\mathcal{E}_5^{(4)}$ and $\mathcal{E}_5^{(5)}$. The most efficient way to compute the H-representations of the cones has been using the software \texttt{polymake} \cite{polymake}  and its algorithm \href{https://polymake.org/release_docs/3.0/polytope.html#minkowski_sum__Polytope__Polytope}{\texttt{minkowski\_sum}}. Precisely, the cones in the filtration 
have 10, 10, 12, 11 and 10 facets. The cone $\mathcal{E}_5$ is defined by the following 10 inequalities
\begin{align*}
& w_{1, 2}+w_{1, 3}+w_{1, 4}+w_{1, 5} \geqslant 0,
&& w_{1, 2}+w_{1, 3}+w_{1, 4}+w_{2, 3}+w_{2, 4}+w_{3, 4}  \geqslant 0,\\ 
& w_{1, 2}+w_{2, 3}+w_{2, 4}+w_{2, 5} \geqslant 0,
&& w_{1, 2}+w_{1, 3}+w_{1, 5}+w_{2, 3}+w_{2, 5}+w_{3, 5}  \geqslant 0,\\
& w_{1, 3}+w_{2, 3}+w_{3, 4}+w_{3, 5}  \geqslant 0,
&& w_{1, 2}+w_{1, 4}+w_{1, 5}+w_{2, 4}+w_{2, 5}+w_{4, 5}  \geqslant 0,\\
& w_{1, 4}+w_{2, 4}+w_{3, 4}+w_{4, 5}  \geqslant 0,
&& w_{1, 3}+w_{1, 4}+w_{1, 5}+w_{3, 4}+w_{3, 5}+w_{4, 5}  \geqslant 0,\\
& w_{1, 5}+w_{2, 5}+w_{3, 5}+w_{4, 5}  \geqslant 0,
&& w_{2, 3}+w_{2, 4}+w_{2, 5}+w_{3,4}+w_{3, 5}+w_{4, 5}  \geqslant 0,
\end{align*}
 and it is easy to check that it contains $\mathcal{F}_5$.
\end{Example}

\normalsize
Since the large dimension of the ambient vector space $\boldsymbol{W}_n$ represents a major obstacle in the effective computation of V- and H-representations of the polyhedral cones involved, we now move to the quotient space $\boldsymbol{W}_n/\boldsymbol{V}_n$.

\begin{Lemma}
Let $\phi_n : \boldsymbol{W}_n \to \boldsymbol{W}_n/\boldsymbol{V}_n$ be the projection map. Then,
\begin{equation}\label{eq:twoWays}
\mathcal{F}_n = \phi_n^{-1}\big(\phi_n(\mathcal{F}_n)\big) \qquad \text{and}\qquad
\mathcal{E}_n = \phi_n^{-1}\big(\phi_n(\mathcal{E}_n)\big).
\end{equation}
\end{Lemma}
\begin{proof}
The property that $A = \phi_n^{-1}\big(\phi_n(A)\big)$ characterizes the subsets $A$ of $\boldsymbol{W}_n$ that are invariant under translations by elements of $\boldsymbol{V}_n$, i.e.~such that
\[
\forall\ \boldsymbol{a}\in A,\ \forall\ \boldsymbol{v} \in \boldsymbol{V}_n \quad \Longrightarrow\quad \boldsymbol{a}+\boldsymbol{v} \in A. 
\]
The cone $\mathcal{E}_n$ is invariant under translations by elements of $\boldsymbol{V}_n$ by definition and $\mathcal{F}_n$ is invariant because linear equivalence of divisors implies numerical equivalence.
\end{proof}

\setcounter{Conjecture}{0}
\begin{Conjecture}[Convex geometry formulation C]
The F-nef cone $\overline{\mathcal{F}_n} := \phi_n(\mathcal{F}_n)$ in $\Pic(\monbar)_{\QQ}$ is contained in the cone $\overline{\mathcal{E}_n} := \phi_n(\mathcal{E}_n)$.
\end{Conjecture}

We immediately notice that the cone $\overline{\mathcal{E}_n}$ can be obtained as projection of the positive orthant $\mathcal{O}_n \subset \boldsymbol{W}_n$. Indeed, by definition the preimage of $\phi_n(\mathcal{O}_n)$ is the smallest subspace containing $\mathcal{O}_n$ invariant under translations by elements of $\boldsymbol{V}_n$, and by construction this subspace is $\mathcal{O}_n + \boldsymbol{V}_n = \mathcal{E}_n$. 


A nice way to describe the cone $\overline{\mathcal{F}_n}$ is to consider a vector subspace $\boldsymbol{U} \subset \boldsymbol{W}_n$ such that $\boldsymbol{U} \oplus \boldsymbol{V}_n \simeq \boldsymbol{W}_n$ (so that $\boldsymbol{U} \simeq \boldsymbol{W}_n/\boldsymbol{V}_n \simeq \Pic(\monbar)_{\QQ}$). We focus on subspaces $\boldsymbol{U}$ spanned by subsets $\mathscr{U}$ of the classes of boundary components $\{\delta_S\ \vert\ S\in\mathfrak{S}\}$. Let $\mathfrak{U} \subset \mathfrak{S}$ be the set of subsets $S$ indexing the elements in $\mathscr{U}$. In this way, we have
\[
\overline{\mathcal{F}_n} = \phi_n(\mathcal{F}_n) = \mathcal{F}_n \cap \boldsymbol{U} = \mathcal{F}_n \cap \left\{ w_S = 0\ \vert\ \forall\ S \in \mathfrak{S}\setminus\mathfrak{U} \right\}.
\]
At this point, we face again the problem of establishing the containment of the cone $\overline{\mathcal{F}_n}$ described by a H-representation in the cone $\overline{\mathcal{E}_n}$ described by a V-representation, but with a sensible reduction of the dimensions of the ambient spaces and of the number of inequalities defining the cones.

%
%
%
%
%
%

The cone $\overline{\mathcal{E}_n}$ is generated by the vectors $\phi_n(\Delta_S) = \delta_S$, that we divide in two groups based on membership in the basis $\mathscr{U}$ of $\boldsymbol{U}$:
\begin{equation}
\big\{ \delta_S\ \vert\ S \in \mathfrak{U}\big\} \cup \left\{ \delta_S = \sum_{T \in \mathfrak{U}} e_T \delta_T\  \Big\vert\ S \in \mathfrak{S}\setminus\mathfrak{U}\right\}.
\end{equation}
The first set of $\overline{N}$ vectors generates the positive orthant of the subspace $\boldsymbol{U}$ for which we know both the V- and the H-representation. Hence, we use again an incremental procedure. Chosen an ordering $\bE^{(1)},\ldots,\bE^{(M)}$ on the elements of $\{ \delta_S = \sum_{T} e_T \delta_T\  \vert$ $S \in \mathfrak{S}\setminus\mathfrak{U}\}$, we denote by $\overline{\mathcal{E}_n}^{(0)}$ the positive orthant $\QQp\langle \delta_S\ \vert\ S \in \mathfrak{U} \rangle \subset \boldsymbol{U}$ and by $\overline{\mathcal{E}_n}^{(i)}$ the cone $\overline{\mathcal{E}_n}^{(i-1)} + \QQp\langle \bE^{(i)}\rangle$.
We have the filtration
\begin{equation}\label{eq:filtrationBar}
\overline{\mathcal{E}_n}^{(0)} \subset \overline{\mathcal{E}_n}^{(1)} \subset \cdots \subset \overline{\mathcal{E}_n}^{(M)} = \overline{\mathcal{E}_n}
\end{equation}
and we look for the smallest $k$ such that 
\[
\overline{\mathcal{F}_n} \subset \overline{\mathcal{E}_n}^{(k)} = \QQp\big\langle \delta_S\ \big\vert\ S \in \mathfrak{U}\big\rangle + \QQp \big\langle\bE^{(1)},\ldots,\bE^{(k)}\big\rangle \subset \overline{\mathcal{E}_n}.
\]

We introduce a integer index that we use to measure how the cones in the filtration are far from containing the $F$-nef cone $\overline{\mathcal{F}_n}$.
\begin{Definition}
Let $\mathcal{C}$ and $\mathcal{D}$ be two polyhedral cones contained in a same vector space. Let $\bigcap_{j} \{ H_j \geqslant 0\}$ be the H-representation of $\mathcal{D}$. We define the \emph{index of containment} of $\mathcal{C}$ in $\mathcal{D}$ as the integer
\[
\Gamma(\mathcal{C},\mathcal{D}) := \# \left\{ H_j\ \big\vert\ \min H_j([\bC]) = -\infty,\ \bC \in \mathcal{C}\right\}.
\]
\end{Definition}
Notice that if we also know the H-representation $\bigcap_{k} \{ L_k \geqslant 0\}$ of $\mathcal{C}$, the index of containment can be efficiently computed by applying linear programming algorithms. Indeed, we need to count the number of unbounded linear optimizations $\min \{H_j([\bW])\ \vert\ L_k([\bW]) \geqslant 0,\ \forall\ k\}$.


In order to minimize the number of steps $k$ necessary for $\overline{\mathcal{E}_n}^{(k)}$ to contain $\overline{\mathcal{F}_n}$, we try to construct a filtration in which the index of containment $\Gamma(\overline{\mathcal{F}_n},\overline{\mathcal{E}_n}^{(i)})$ is as small as possible at each step. We are free to choose
\begin{enumerate}[1.]
\item the subspace $\boldsymbol{U} = \QQ\langle\delta_S\ \vert\ S\in \mathfrak{U}\rangle$,
\item the order in which we add the elements $\delta_S = \sum_{T\in\mathfrak{U}} e_T \delta_T,\ S \notin \mathfrak{U}$.
\end{enumerate}
The first choice is crucial because the starting cone $\overline{\mathcal{E}_n}^{(0)}$ is the positive orthant of $\boldsymbol{U}$. From computational studies, we found that the minimal index of containment $\Gamma(\overline{\mathcal{F}_n},\overline{\mathcal{E}_n}^{(0)})$ is obtained considering the subspace $\boldsymbol{U} \simeq \Pic(\monbar)_{\QQ}$ generated by the basis $\mathscr{B}_n$ in \eqref{eq:baseFontanari} (see Table \ref{tab:indexOfContainmentOrthant} for a detailed report on the computational experiment).

\begin{table}[!ht]
\begin{center}
\begin{tikzpicture}[scale=0.85]
\draw[thick] (0,0) -- (6,0);
\draw[thick] (-4.5,-1) -- (6,-1);
\draw (-4.5,-2) -- (6,-2);
\draw (-4.5,-3) -- (6,-3);
\draw (-4.5,-4) -- (6,-4);
\draw (-4.5,-5) -- (6,-5);
\draw[thick] (-4.5,-6) -- (6,-6);

\draw[thick] (-4.5,-1) -- (-4.5,-6);
\draw[thick] (0,0) -- (0,-6);
\draw[thick] (2,0) -- (2,-6);
\draw[thick] (4,0) -- (4,-6);
\draw[thick] (6,0) -- (6,-6);
\node at (1,-.5) [] {$n=5$};
\node at (3,-.5) [] {$n=6$};
\node at (5,-.5) [] {$n=7$};

\node at (-2.25,-1.5) [] {\footnotesize $\dim \boldsymbol{V}_n\ /\ \dim \boldsymbol{W}_n$};


\node at (-2.25, -2.5) [] {\footnotesize $\min \Gamma\big(\overline{\mathcal{F}_n},\overline{\mathcal{E}_n}^{(0)}\big)$};

\node at (-2.25, -3.5) [] {\footnotesize $\max \Gamma\big(\overline{\mathcal{F}_n},\overline{\mathcal{E}_n}^{(0)}\big)$};

\node at (-2.25, -4.5) [] {\footnotesize $\text{average } \Gamma\big(\overline{\mathcal{F}_n},\overline{\mathcal{E}_n}^{(0)}\big)$};

\node at (-2.25, -5.5) [] {\footnotesize $\Gamma\big(\overline{\mathcal{F}_n},\QQp\langle \mathscr{B}_n\rangle\big)$};

\node at (1,-1.5) [] {\small $ 5\ /\ 10$};
\node at (1,-2.5) [] {\small $0$};
\node at (1,-3.5) [] {\small $2$};
\node at (1,-4.5) [] {\small $1.\overline{1}$};
\node at (1,-5.5) [] {\small $0$};

\node at (3,-1.5) [] {\small $9 \ /\ 25$};
\node at (3,-2.5) [] {\small $1$};
\node at (3,-3.5) [] {\small $15\ \footnotemark[2]$};
\node at (3,-4.5) [] {\small $7.09\ \footnotemark[2]$};
\node at (3,-5.5) [] {\small $1$};

\node at (5,-1.5) [] {\small $14 \ /\ 56$};
\node at (5,-2.5) [] {\small $7$};
\node at (5,-3.5) [] {\small $ 42\ \footnotemark[2]$};
\node at (5,-4.5) [] {\small $ 30.16\ \footnotemark[2]$};
\node at (5,-5.5) [] {\small $7$};

\end{tikzpicture}
\caption{\label{tab:indexOfContainmentOrthant} Index of containment of the $F$-nef cone $\overline{\mathcal{F}_n}$ in the positive orthant $\overline{\mathcal{E}_n}^{(0)} \subset \boldsymbol{U}$ varying the subspace $\boldsymbol{U} \subset	\boldsymbol{W}_n$ for $n=5,6,7$. The last line shows the index of containment of $\overline{\mathcal{F}_n}$ in the positive orthant in the case of $\boldsymbol{U}$ spanned by the basis $\mathscr{B}_n$.}
\end{center}
\end{table}

\footnotetext[2]{We give an estimate of the maximum and of the average of the index of containment in the case $n=6,7$ based on a subset of 100000 subspaces $\boldsymbol{U}\subset \boldsymbol{W}_n$ of dimension $\overline{N}$ randomly chosen among the whole set of $\binom{N}{M}$ possibilities.}

Regarding the second choice, we remark that if $\min \{ H_j([\bA])\ \vert\ \bA \in \overline{\mathcal{F}_n}\} = -\infty$, then a piece of $\overline{\mathcal{F}_n}$ lies in the half-space $H_j < 0$. Now, assume that the index of containment $\Gamma(\overline{\mathcal{F}_n},\overline{\mathcal{E}_n}^{(i-1)})$ is $\gamma$, namely there are $\gamma$ half-spaces $H_j < 0$ containing a portion of the $F$-nef cone. We need to enlarge the cone in all such directions and we choose a vector $\delta_S,\ S \notin \mathfrak{U}$ such that for any other vector $\delta_{S'},\ S' \notin \mathfrak{U}$ the number of inequalities $H_j < 0$ satisfied by $\delta_{S'}$ is at most the number of inequalities satisfied by $\delta_S$.

\section{Main results}

In this section we prove Conjecture \ref{main} for $n=7$ and $n=6$. The code used in the computation is available at the webpage \href{http://www.paololella.it/EN/Publications.html}{\texttt{www.paololella.it/EN/} \texttt{Publications.html}}. It is based on the the simultaneous use of software \emph{Macaulay2} \cite{M2}, \texttt{polymake} \cite{polymake} and \texttt{lpSolve} \cite{lpSolve} (through the \texttt{R} interface).

\subsection{\texorpdfstring{Case $n=7$}{Case n = 7}}

The vector space $\boldsymbol{W}_7$ is generated by the boundary components
\[
\begin{split}
&\Delta_{1, 2},\, \Delta_{1, 2, 3},\, \Delta_{1, 2, 4},\, \Delta_{1, 2, 5},\, \Delta_{1, 2, 6},\, \Delta_{1, 2, 7},\, \Delta_{1, 3},\, \Delta_{1, 3, 4},\, \Delta_{1, 3, 5},\, \Delta_{1, 3, 6}, \\
&\Delta_{1, 3, 7},\, \Delta_{1, 4},\, \Delta_{1, 4, 5},\, \Delta_{1, 4, 6},\, \Delta_{1, 4, 7},\, \Delta_{1, 5},\, \Delta_{1, 5, 6},\, \Delta_{1, 5, 7},\, \Delta_{1, 6},\, \Delta_{1, 6, 7},\\
& \Delta_{1, 7},\, \Delta_{2, 3},\, \Delta_{2, 3, 4},\, \Delta_{2, 3, 5},\, \Delta_{2, 3, 6},\, \Delta_{2, 3, 7},\, \Delta_{2, 4},\, \Delta_{2, 4, 5},\, \Delta_{2, 4, 6},\, \Delta_{2, 4, 7},\\
& \Delta_{2, 5},\, \Delta_{2, 5, 6},\, \Delta_{2, 5, 7},\, \Delta_{2, 6},\, \Delta_{2, 6, 7},\, \Delta_{2, 7},\, \Delta_{3, 4},\, \Delta_{3, 4, 5},\, \Delta_{3, 4, 6},\, \Delta_{3, 4, 7},\\
& \Delta_{3, 5},\, \Delta_{3, 5, 6},\, \Delta_{3, 5, 7},\, \Delta_{3, 6},\, \Delta_{3, 6, 7},\, \Delta_{3, 7},\, \Delta_{4, 5},\, \Delta_{4, 5, 6},\, \Delta_{4, 5, 7},\, \Delta_{4, 6},\\
& \Delta_{4, 6, 7},\, \Delta_{4, 7},\, \Delta_{5, 6},\, \Delta_{5, 6, 7},\, \Delta_{5, 7},\, \Delta_{6, 7}.
\end{split}
\] 
The subspace $\boldsymbol{V}_7$ has dimension 14 and the $F$-nef cone $\mathcal{F}_7$ is defined by 350 linear inequalities. In this case, we could not compute all cones of the filtration \eqref{eq:filtrationE}. In reasonable time, we can obtain the H-representation of the first eight cones in the filtration. In Table \ref{tab:incremental_N7}\subref{tab:filtrationN7}, there is the description of the cones obtained from the subset of a basis of $\boldsymbol{V}_7$ composed by vectors 
\[
\begin{split}
 \bV^{(1)} ={}& \Delta_{1, 2}+\Delta_{1, 2, 5}+\Delta_{1, 2, 6}+\Delta_{1, 2, 7}-\Delta_{1, 3}-\Delta_{1, 3, 5}-\Delta_{1, 3, 6}-\Delta_{1, 3, 7}\\
 			&-\Delta_{2, 4}-\Delta_{2, 4, 5}-\Delta_{2, 4, 6}-\Delta_{2, 4, 7}+\Delta_{3, 4}+\Delta_{3, 4, 5}+\Delta_{3, 4, 6}+\Delta_{3, 4, 7},\\
\bV^{(2)} ={} & \Delta_{1, 2}+\Delta_{1, 2, 5}+\Delta_{1, 2, 6}+\Delta_{1, 2, 7}-\Delta_{1, 4}-\Delta_{1, 4, 5}-\Delta_{1, 4, 6}-\Delta_{1, 4, 7}\\
			&-\Delta_{2, 3}-\Delta_{2, 3, 5}-\Delta_{2, 3, 6}-\Delta_{2, 3, 7}+\Delta_{3, 4}+\Delta_{3, 4, 5}+\Delta_{3, 4, 6}+\Delta_{3, 4, 7},\\
\end{split}
\]
\[
\begin{split}
\bV^{(3)} ={} & \Delta_{1, 2}+\Delta_{1, 2, 4}+\Delta_{1, 2, 6}+\Delta_{1, 2, 7}-\Delta_{1, 3}-\Delta_{1, 3, 4}-\Delta_{1, 3, 6}-\Delta_{1, 3, 7}\\
			&-\Delta_{2, 4, 5}-\Delta_{2, 5}-\Delta_{2, 5, 6}-\Delta_{2, 5, 7}+\Delta_{3, 4, 5}+\Delta_{3, 5}+\Delta_{3, 5, 6}+\Delta_{3, 5, 7},\\
\bV^{(4)} ={} & \Delta_{1, 2}+\Delta_{1, 2, 4}+\Delta_{1, 2, 6}+\Delta_{1, 2, 7}-\Delta_{1, 4, 5}-\Delta_{1, 5}-\Delta_{1, 5, 6}-\Delta_{1, 5, 7}\\
			&-\Delta_{2, 3}-\Delta_{2, 3, 4}-\Delta_{2, 3, 6}-\Delta_{2, 3, 7}+\Delta_{3, 4, 5}+\Delta_{3, 5}+\Delta_{3, 5, 6}+\Delta_{3, 5, 7},\\
\bV^{(5)} ={} & \Delta_{1, 2}+\Delta_{1, 2, 4}+\Delta_{1, 2, 5}+\Delta_{1, 2, 7}-\Delta_{1, 3}-\Delta_{1, 3, 4}-\Delta_{1, 3, 5}-\Delta_{1, 3, 7}\\
			&-\Delta_{2, 4, 6}-\Delta_{2, 5, 6}-\Delta_{2, 6}-\Delta_{2, 6, 7}+\Delta_{3, 4, 6}+\Delta_{3, 5, 6}+\Delta_{3, 6}+\Delta_{3, 6, 7},\\
\bV^{(6)} ={} & \Delta_{1, 2}+\Delta_{1, 2, 4}+\Delta_{1, 2, 5}+\Delta_{1, 2, 7}-\Delta_{1, 4, 6}-\Delta_{1, 5, 6}-\Delta_{1, 6}-\Delta_{1, 6, 7}\\
			&-\Delta_{2, 3}-\Delta_{2, 3, 4}-\Delta_{2, 3, 5}-\Delta_{2, 3, 7}+\Delta_{3, 4, 6}+\Delta_{3, 5, 6}+\Delta_{3, 6}+\Delta_{3, 6, 7},\\
\bV^{(7)} ={} & \Delta_{1, 2}+\Delta_{1, 2, 4}+\Delta_{1, 2, 5}+\Delta_{1, 2, 6}-\Delta_{1, 3}-\Delta_{1, 3, 4}-\Delta_{1, 3, 5}-\Delta_{1, 3, 6}\\
&-\Delta_{2, 4, 7}-\Delta_{2, 5, 7}-\Delta_{2, 6, 7}-\Delta_{2, 7}+\Delta_{3, 4, 7}+\Delta_{3, 5, 7}+\Delta_{3, 6, 7}+\Delta_{3, 7}.
\end{split}
\]
Notice that the cone $\mathcal{E}_7^{(7)}$ is defined by 99281 inequalities and the index of containment is still very large, namely only 32 inequalities of the H-representation of $\mathcal{E}_7^{(7)}$ are satisfied by all elements of $\mathcal{F}_7$. 

Thus, we consider the space $\boldsymbol{U} \simeq \Pic(\overline{M}_{0,7})_{\QQ}$ spanned by the classes of boundary components $\mathscr{B}_7$
\begin{align*}
& \delta_{1, 2, 5}, && \delta_{1, 2, 6}, && \delta_{1, 2, 7}, && \delta_{1, 3, 4}, && \delta_{1, 3, 6}, && \delta_{1, 3, 7}, && \delta_{1, 4}, \\
& \delta_{1, 4, 5}, && \delta_{1, 4, 6}, && \delta_{1, 4, 7}, && \delta_{1, 5}, && \delta_{1, 5, 6}, && \delta_{1, 5, 7}, && \delta_{1, 6}, \\
& \delta_{1, 6, 7}, && \delta_{1, 7}, && \delta_{2, 3}, && \delta_{2, 3, 4}, && \delta_{2, 3, 5}, && \delta_{2, 3, 6}, && \delta_{2, 3, 7},\\
& \delta_{2, 4, 5}, && \delta_{2, 4, 7}, && \delta_{2, 5}, && \delta_{2, 5, 6}, && \delta_{2, 5, 7}, && \delta_{2, 6}, && \delta_{2, 6, 7},\\
& \delta_{2, 7}, && \delta_{3, 4}, && \delta_{3, 4, 5}, && \delta_{3, 4, 6}, && \delta_{3, 4, 7}, && \delta_{3, 5, 6}, && \delta_{3, 6},\\
& \delta_{3, 6, 7}, && \delta_{3, 7}, && \delta_{4, 5}, && \delta_{4, 5, 6}, && \delta_{4, 5, 7}, && \delta_{4, 7}, && \delta_{5, 6}.
\end{align*}
The index of containment $\Gamma(\overline{\mathcal{F}_7},\overline{\mathcal{E}_7}^{(0)})$ is equal to $7$ and the inequalities not satisfied are $w_{1, 4, 5} \geq 0$, $w_{1, 4, 7} \geq 0$, $w_{1, 5, 6} \geq 0$, $w_{2, 3, 6} \geq 0$, $w_{2, 3, 7} \geq 0$, $w_{2, 5, 6} \geq 0$ and $w_{3, 4, 7} \geq 0$. Now, we construct the filtration \eqref{eq:filtrationBar} adding at each step the vector $\delta_S \notin \mathscr{B}_7$ that enlarges the cone $\overline{\mathcal{E}_7}^{(i)}$ mostly. For instance, the vectors $\delta_S \notin \mathscr{B}_7$ lie outside 1 or 3 half-spaces among the 7 listed above, so we pick a vector among those enlarging the positive orthant in three of the needed directions. This procedure leads to add the following vectors (in this order)
\[
\begin{split}
\delta_{1, 2, 3} ={}& \delta_{1, 2, 5}+\delta_{1, 3, 4}-\delta_{1, 4, 5}-\delta_{2, 3}-\delta_{2, 3, 6}-\delta_{2, 3, 7}+\delta_{2, 5}+\delta_{2, 5, 6}\\
			&+\delta_{2, 5, 7}+\delta_{3, 4}+\delta_{3, 4, 6}+\delta_{3, 4, 7}-\delta_{4, 5}-\delta_{4, 5, 6}-\delta_{4, 5, 7},\\
\delta_{4, 6, 7} ={}& -\delta_{1, 2, 5}+\delta_{1, 2, 7}+\delta_{1, 4, 5}-\delta_{1, 4, 7}-\delta_{2, 3, 5}+\delta_{2, 3, 7}-\delta_{2, 5}-\delta_{2, 5, 6}\\
			&+\delta_{2, 6, 7}+\delta_{2, 7}+\delta_{3, 4, 5}-\delta_{3, 4, 7}+\delta_{4, 5}+\delta_{4, 5, 6}-\delta_{4, 7},\\
 \delta_{1, 3, 5} ={}& \delta_{1, 3, 7}-\delta_{1, 4, 5}+\delta_{1, 4, 7}-\delta_{1, 5}-\delta_{1, 5, 6}+\delta_{1, 6, 7}+\delta_{1, 7}+\delta_{2, 3, 5}\\
 			&-\delta_{2, 3, 7}+\delta_{2, 4, 5}-\delta_{2, 4, 7}+\delta_{2, 5}+\delta_{2, 5, 6}-\delta_{2, 6, 7}-\delta_{2, 7},\\
			 \delta_{2, 4, 6} ={}& \delta_{1, 3, 6}-\delta_{1, 3, 7}+\delta_{1, 4, 6}-\delta_{1, 4, 7}+\delta_{1, 5, 6}-\delta_{1, 5, 7}+\delta_{1, 6}-\delta_{1, 7}\\
 			&-\delta_{2, 3, 6}+\delta_{2, 3, 7}+\delta_{2, 4, 7}-\delta_{2, 5, 6}+\delta_{2, 5, 7}-\delta_{2, 6}+\delta_{2, 7},\\
 \delta_{3, 5, 7} ={}&  -\delta_{1, 2, 6}+\delta_{1, 2, 7}-\delta_{1, 4, 6}+\delta_{1, 4, 7}-\delta_{1, 5, 6}+\delta_{1, 5, 7}-\delta_{1, 6}+\delta_{1, 7}\\
 			&+\delta_{2, 3, 6}-\delta_{2, 3, 7}+\delta_{3, 4, 6}-\delta_{3, 4, 7}+\delta_{3, 5, 6}+\delta_{3, 6}-\delta_{3, 7}.
 \end{split}
\]
In particular, $\delta_{1, 2, 3}$ enlarges 3 of the 7 half-spaces of $\overline{\mathcal{E}_7}^{(0)}$ not containing $\overline{\mathcal{F}_7}$, $\delta_{4, 6, 7}$ enlarges 8 of the 14 half-spaces of $\overline{\mathcal{E}_7}^{(1)}$ not containing $\overline{\mathcal{F}_7}$, $\delta_{1, 3, 5}$ enlarges 13 of the 16 half-spaces of $\overline{\mathcal{E}_7}^{(2)}$ not containing $\overline{\mathcal{F}_7}$, $\delta_{2, 4, 6}$ enlarges 6 of the 8 half-spaces of $\overline{\mathcal{E}_7}^{(3)}$ not containing $\overline{\mathcal{F}_7}$ and $\delta_{3, 5, 7}$ enlarges all 4 half-spaces of $\overline{\mathcal{E}_7}^{(4)}$ not containing $\overline{\mathcal{F}_7}$. Finally, the cone $\overline{\mathcal{E}_7}^{(5)}$ contains the $F$-nef cone $\overline{\mathcal{F}_7}$.

\begin{Theorem}\label{prop:case7formC}
The cone $\overline{\mathcal{F}_7}$ is contained in the cone $\overline{\mathcal{E}_7}$.
\end{Theorem}

\bigskip

\begin{table}[!ht]
\begin{center}
\subfloat[][The first 8 steps of the filtration \eqref{eq:filtrationE} in the case $n=7$.]
{
\label{tab:filtrationN7}
\begin{tikzpicture}[scale=0.85]
\draw (-4,0) -- (-4,-3);
\draw (0,1) -- (0,-3);
\draw (1.5,1) -- (1.5,-3);
\draw (3,1) -- (3,-3);
\draw (4.5,1) -- (4.5,-3);
\draw (6,1) -- (6,-3);
\draw (7.5,1) -- (7.5,-3);
\draw (9,1) -- (9,-3);
\draw (10.5,1) -- (10.5,-3);
\draw (12,1) -- (12,-3);

\draw (0,1) -- (12,1);
\draw (-4,0) -- (12,0);
\draw (-4,-1) -- (12,-1);
\draw (-4,-2) -- (12,-2);
\draw (-4,-3) -- (12,-3);

\node at (-2,-0.5) [] {\small number of facets};
\node at (-2,-1.5) [] {\small $\Gamma\big(\mathcal{F}_7,\mathcal{E}_7^{(i)}\big)$};
\node at (-2,-2.5) [] {\small computing time\footnotemark[3]};

\node at (0.75,0.5) [] {\small $\mathcal{E}_7^{(0)}$};
\node at (2.25,0.5) [] {\small $\mathcal{E}_7^{(1)}$};
\node at (3.75,0.5) [] {\small $\mathcal{E}_7^{(2)}$};
\node at (5.25,0.5) [] {\small $\mathcal{E}_7^{(3)}$};
\node at (6.75,0.5) [] {\small $\mathcal{E}_7^{(4)}$};
\node at (8.25,0.5) [] {\small $\mathcal{E}_7^{(5)}$};
\node at (9.75,0.5) [] {\small $\mathcal{E}_7^{(6)}$};
\node at (11.25,0.5) [] {\small $\mathcal{E}_7^{(7)}$};

\node at (0.75,-0.5) [] {\footnotesize 56};
\node at (2.25,-0.5) [] {\footnotesize 104};
\node at (3.75,-0.5) [] {\footnotesize 544};
\node at (5.25,-0.5) [] {\footnotesize 1320};
\node at (6.75,-0.5) [] {\footnotesize 4052};
\node at (8.25,-0.5) [] {\footnotesize 12276};
\node at (9.75,-0.5) [] {\footnotesize 28966};
\node at (11.25,-0.5) [] {\footnotesize 99281};

\node at (0.75,-1.5) [] {\footnotesize 56};
\node at (2.25,-1.5) [] {\footnotesize 104};
\node at (3.75,-1.5) [] {\footnotesize 544};
\node at (5.25,-1.5) [] {\footnotesize 1320};
\node at (6.75,-1.5) [] {\footnotesize 4052};
\node at (8.25,-1.5) [] {\footnotesize 12276};
\node at (9.75,-1.5) [] {\footnotesize 28966};
\node at (11.25,-1.5) [] {\footnotesize 99249};

\node at (0.75,-2.5) [] {\footnotesize };
\node at (2.25,-2.5) [] {\footnotesize 3.6s};
\node at (3.75,-2.5) [] {\footnotesize 4.1s};
\node at (5.25,-2.5) [] {\footnotesize 5.7s};
\node at (6.75,-2.5) [] {\footnotesize 14.1s};
\node at (8.25,-2.5) [] {\footnotesize 45.7s};
\node at (9.75,-2.5) [] {\footnotesize 153.3s};
\node at (11.25,-2.5) [] {\footnotesize 2034.4s};

\end{tikzpicture}
}

\subfloat[][The first 8 steps of the filtration \eqref{eq:filtrationBar} in the case $n=7$.]
{
\label{tab:filtrationN7bar}
\begin{tikzpicture}[scale=0.85]
\draw (-4,0) -- (-4,-3);
\draw (0,1) -- (0,-3);
\draw (1.5,1) -- (1.5,-3);
\draw (3,1) -- (3,-3);
\draw (4.5,1) -- (4.5,-3);
\draw (6,1) -- (6,-3);
\draw (7.5,1) -- (7.5,-3);
\draw (9,1) -- (9,-3);
\draw (10.5,1) -- (10.5,-3);
\draw (12,1) -- (12,-3);

\draw (0,1) -- (12,1);
\draw (-4,0) -- (12,0);
\draw (-4,-1) -- (12,-1);
\draw (-4,-2) -- (12,-2);
\draw (-4,-3) -- (12,-3);

\node at (-2,-0.5) [] {\small number of facets};
\node at (-2,-1.5) [] {\small $\Gamma\big(\overline{\mathcal{F}_7},\overline{\mathcal{E}_7}^{(i)}\big)$};
\node at (-2,-2.5) [] {\small computing time\footnotemark[3]};

\node at (0.75,0.5) [] {\small $\overline{\mathcal{E}_7}^{(0)}$};
\node at (2.25,0.5) [] {\small $\overline{\mathcal{E}_7}^{(1)}$};
\node at (3.75,0.5) [] {\small $\overline{\mathcal{E}_7}^{(2)}$};
\node at (5.25,0.5) [] {\small $\overline{\mathcal{E}_7}^{(3)}$};
\node at (6.75,0.5) [] {\small $\overline{\mathcal{E}_7}^{(4)}$};
\node at (8.25,0.5) [] {\small $\overline{\mathcal{E}_7}^{(5)}$};
\node at (9.75,0.5) [] {\small $\overline{\mathcal{E}_7}^{(6)}$};
\node at (11.25,0.5) [] {\small $\overline{\mathcal{E}_7}^{(7)}$};

\node at (0.75,-0.5) [] {\footnotesize 42};
\node at (2.25,-0.5) [] {\footnotesize 91};
\node at (3.75,-0.5) [] {\footnotesize 196};
\node at (5.25,-0.5) [] {\footnotesize 477};
\node at (6.75,-0.5) [] {\footnotesize 1433};
\node at (8.25,-0.5) [] {\footnotesize 5753};
\node at (9.75,-0.5) [] {\footnotesize 22996};
\node at (11.25,-0.5) [] {\footnotesize 69929};

\node at (0.75,-1.5) [] {\footnotesize 7};
\node at (2.25,-1.5) [] {\footnotesize 14};
\node at (3.75,-1.5) [] {\footnotesize 16};
\node at (5.25,-1.5) [] {\footnotesize 8};
\node at (6.75,-1.5) [] {\footnotesize 4};
\node at (8.25,-1.5) [] {\footnotesize 0};
\node at (9.75,-1.5) [] {\footnotesize 0};
\node at (11.25,-1.5) [] {\footnotesize 0};

\node at (0.75,-2.5) [] {\footnotesize };
\node at (2.25,-2.5) [] {\footnotesize 2.7s};
\node at (3.75,-2.5) [] {\footnotesize 2.8s};
\node at (5.25,-2.5) [] {\footnotesize 3.4s};
\node at (6.75,-2.5) [] {\footnotesize 5.3s};
\node at (8.25,-2.5) [] {\footnotesize 15.9s};
\node at (9.75,-2.5) [] {\footnotesize 98.2s};
\node at (11.25,-2.5) [] {\footnotesize 1568.4s};

\end{tikzpicture}
}
\caption{\label{tab:incremental_N7} The initial part of the filtrations of cones $\mathcal{E}_7^{(i)}$ in $\boldsymbol{W}_7$ and $\overline{\mathcal{E}_7}^{(i)}$ in $\boldsymbol{U} = \QQ\langle\mathscr{B}_7\rangle$.}
\end{center}
\end{table}

\footnotetext[3]{The computation has been run on a MacBook Pro with an Intel Core 2 Duo 2.4 GHz processor.}	

\subsection{\texorpdfstring{Case $n=6$}{Case n = 6}}
The vector space $\boldsymbol{W}_6$ is generated by the boundary components
\[
\begin{split}
&\Delta_{1, 2},\, \Delta_{1, 2, 3},\, \Delta_{1, 2, 4},\, \Delta_{1, 2, 5},\, \Delta_{1, 2, 6},\, \Delta_{1, 3},\,
      \Delta_{1, 3, 4},\, \Delta_{1, 3, 5},\\
& \Delta_{1, 3, 6},\, \Delta_{1, 4},\, \Delta_{1, 4, 5},\, \Delta_{1, 4, 6},\, \Delta_{1, 5},\, \Delta_{1, 5, 6},\, \Delta_{1, 6},\, \Delta_{2, 3},\\
& \Delta_{2, 4},\, \Delta_{2, 5},\, \Delta_{2, 6},\, \Delta_{3, 4},\, \Delta_{3, 5},\, \Delta_{3, 6},\, \Delta_{4, 5},\, \Delta_{4, 6},\, \Delta_{5, 6},
\end{split}
\]
the subspace $\boldsymbol{V}_6$ is spanned by 
\[
\begin{split}
&\bV^{(1)} = \Delta_{1, 2}+\Delta_{1, 2, 5}+\Delta_{1, 2, 6}-\Delta_{1, 3}-\Delta_{1, 3, 5}-\Delta_{1, 3, 6}-\Delta_{2, 4}+\Delta_{3, 4},\\
&\bV^{(2)} = \Delta_{1, 2}+\Delta_{1, 2, 5}+\Delta_{1, 2, 6}-\Delta_{1, 4}-\Delta_{1, 4, 5}-\Delta_{1, 4, 6}-\Delta_{2, 3}+\Delta_{3, 4},\\
&\bV^{(3)} = \Delta_{1, 2}+\Delta_{1, 2, 4}+\Delta_{1, 2, 6}-\Delta_{1, 3}-\Delta_{1, 3, 4}-\Delta_{1, 3, 6}-\Delta_{2, 5}+\Delta_{3, 5},\\
&\bV^{(4)} = \Delta_{1, 2}+\Delta_{1, 2, 4}+\Delta_{1, 2, 6}-\Delta_{1, 4, 5}-\Delta_{1, 5}-\Delta_{1, 5, 6}-\Delta_{2, 3}+\Delta_{3, 5},\\
\end{split}
\]
\[
\begin{split}
&\bV^{(5)} =  \Delta_{1, 2}+\Delta_{1, 2, 4}+\Delta_{1, 2, 5}-\Delta_{1, 3}-\Delta_{1, 3, 4}-\Delta_{1, 3, 5}-\Delta_{2, 6}+\Delta_{3, 6},\\
&\bV^{(6)} =  \Delta_{1, 2}+\Delta_{1, 2, 4}+\Delta_{1, 2, 5}-\Delta_{1, 4, 6}-\Delta_{1, 5, 6}-\Delta_{1, 6}-\Delta_{2, 3}+\Delta_{3, 6},\\
&\bV^{(7)} = \Delta_{1, 2}+\Delta_{1, 2, 3}+\Delta_{1, 2, 6}-\Delta_{1, 3, 4}-\Delta_{1, 4}-\Delta_{1, 4, 6}-\Delta_{2, 5}+\Delta_{4, 5},\\
&\bV^{(8)} =  \Delta_{1, 2}+\Delta_{1, 2, 3}+\Delta_{1, 2, 5}-\Delta_{1, 3, 4}-\Delta_{1, 4}-\Delta_{1, 4, 5}-\Delta_{2, 6}+\Delta_{4, 6},\\
&\bV^{(9)} = \Delta_{1, 2}+\Delta_{1, 2, 3}+\Delta_{1, 2, 4}-\Delta_{1, 3, 5}-\Delta_{1, 4, 5}-\Delta_{1, 5}-\Delta_{2, 6}+\Delta_{5, 6},
\end{split}
\]
and the H-representation of the $F$-nef cone $\mathcal{F}_6$ consists of 65 half-spaces. In this case, we can compute the entire filtration 
\[
\mathcal{O}_6 = \mathcal{E}_6^{(0)} \subset \mathcal{E}_6^{(1)} \subset \quad \cdots\quad\subset \mathcal{E}_6^{(8)} \subset \mathcal{E}_6^{(9)} = \mathcal{E}_6
\]
(see Table \ref{tab:incremental_N6}\subref{tab:filtrationN6} for details). Looking at the indices of containment $\Gamma(\mathcal{F}_6,\mathcal{E}_6^{(i)})$, we notice that the first half-space defining a cone $\mathcal{E}_6^{(i)}$ and containing the entire $F$-nef cone appears in the fifth step of the enlargement and that the whole cone $\mathcal{E}_6$ is the unique cone in the filtration containing the cone $\mathcal{F}_6$. 

\begin{Proposition}\label{prop:case6formB}
The cone $\mathcal{F}_6$ is contained in the cone $\mathcal{E}_6$.
\end{Proposition}


The basis $\mathscr{B}_6$ consists of the classes of boundary divisors
\begin{align*}
&\delta_{1, 2, 5}, && \delta_{1, 2, 6},&& \delta_{1, 3, 4},&& \delta_{1, 3, 6},&& \delta_{1, 4},&& \delta_{1, 4, 5},&& \delta_{1, 4, 6},&& \delta_{1, 5},\\
&\delta_{1, 5, 6},&& \delta_{1, 6},&& \delta_{2, 3},&& \delta_{2, 5},&& \delta_{2, 6},&&
      \delta_{3, 4},&& \delta_{3, 6},&& \delta_{4, 5},
\end{align*}
and the remaining 9 classes can be written as follows
\[
\begin{split}
 &\delta_{1, 2} = -\delta_{1, 2, 5}-\delta_{1, 2, 6}+\delta_{1, 4}+\delta_{1, 4, 5}+\delta_{1, 4, 6}+\delta_{2, 3}-\delta_{3, 4},\\
&\delta_{1, 2, 3} = \delta_{1, 2, 5}+\delta_{1, 3, 4}-\delta_{1, 4, 5}-\delta_{2, 3}+\delta_{2, 5}+\delta_{3, 4}-\delta_{4, 5},\\
&\delta_{1, 2, 4} = \delta_{1, 2, 6}-\delta_{1, 4}-\delta_{1, 4, 5}+\delta_{1, 5, 6}+\delta_{1, 6}+\delta_{3, 4}-\delta_{3, 6},\\
& \delta_{1, 3} = -\delta_{1, 3, 4}-\delta_{1, 3, 6}+\delta_{1, 4, 5}+\delta_{1, 5}+\delta_{1, 5, 6}+\delta_{2, 3}-\delta_{2, 5},\\
& \delta_{1, 3, 5} = \delta_{1, 3, 6}-\delta_{1, 4, 5}+\delta_{1, 4, 6}-\delta_{1, 5}+\delta_{1, 6}+\delta_{2, 5}-\delta_{2, 6},\\
& \delta_{2, 4} = \delta_{1, 3, 4}-\delta_{1, 3, 6}+\delta_{1, 4}+\delta_{1, 4, 5}-\delta_{1, 5, 6}-\delta_{1, 6}+\delta_{2, 6},\\
& \delta_{3, 5} = \delta_{1, 2, 5}-\delta_{1, 2, 6}+\delta_{1, 4, 5}-\delta_{1, 4, 6}+\delta_{1, 5}-\delta_{1, 6}+\delta_{3, 6},\\
& \delta_{4, 6} = -\delta_{1, 2, 5}+\delta_{1, 2, 6}+\delta_{1, 4, 5}-\delta_{1, 4, 6}-\delta_{2, 5}+\delta_{2, 6}+\delta_{4, 5},\\
&  \delta_{5, 6} = -\delta_{1, 3, 4}+\delta_{1, 3, 6}+\delta_{1, 4, 5}-\delta_{1, 5, 6}-\delta_{3, 4}+\delta_{3, 6}+\delta_{4, 5}.
\end{split}
\]
The index of containment $\Gamma(\overline{\mathcal{F}_6},\overline{\mathcal{E}_6}^{(0)})$ is equal to 1 and the inequality not satisfied by all elements of $\overline{\mathcal{F}_6}$ is $w_{1, 4, 5} \geqslant 0$. Hence, we start enlarging the positive orthant of $\boldsymbol{U}$ by adding a vector $\delta_S \notin \mathscr{B}_6$ such that $w_{1, 4, 5}(\delta_S) < 0$. There are three possible choices, $\delta_{1, 2, 3}$, $\delta_{1, 2, 4}$ and $\delta_{1, 3, 5}$, and all of them produce a cone $\overline{\mathcal{E}_6}^{(1)} = \QQp\langle\mathscr{B}_6\rangle + \QQp\langle\delta_S\rangle$ with 25 facets containing the $F$-nef cone $\overline{\mathcal{F}_6}$ (see Table \ref{tab:incremental_N6}\subref{tab:filtrationBar}).

\begin{Proposition}\label{prop:case6formC}
The cone $\overline{\mathcal{F}_6}$ is contained in the cone $\overline{\mathcal{E}_6}$.
\end{Proposition}

\begin{table}[!ht]
\begin{center}
\subfloat[][The filtration \eqref{eq:filtrationE} in the case $n=6$.]
{
\label{tab:filtrationN6}
\begin{tikzpicture}[scale=0.875]
\draw (-4,0) -- (-4,-3);
\draw (0,1) -- (0,-3);
\draw (1,1) -- (1,-3);
\draw (2,1) -- (2,-3);
\draw (3,1) -- (3,-3);
\draw (4,1) -- (4,-3);
\draw (5,1) -- (5,-3);
\draw (6,1) -- (6,-3);
\draw (7,1) -- (7,-3);
\draw (8,1) -- (8,-3);
\draw (9,1) -- (9,-3);
\draw (10,1) -- (10,-3);

\draw (0,1) -- (10,1);
\draw (-4,0) -- (10,0);
\draw (-4,-1) -- (10,-1);
\draw (-4,-2) -- (10,-2);
\draw (-4,-3) -- (10,-3);

\node at (-2,-0.5) [] {\small number of facets};
\node at (-2,-1.5) [] {\small $\Gamma\big(\mathcal{F}_6,\mathcal{E}_6^{(i)}\big)$};
\node at (-2,-2.5) [] {\small computing time\footnotemark[3]};

\node at (0.5,0.5) [] {\small $\mathcal{E}_6^{(0)}$};
\node at (1.5,0.5) [] {\small $\mathcal{E}_6^{(1)}$};
\node at (2.5,0.5) [] {\small $\mathcal{E}_6^{(2)}$};
\node at (3.5,0.5) [] {\small $\mathcal{E}_6^{(3)}$};
\node at (4.5,0.5) [] {\small $\mathcal{E}_6^{(4)}$};
\node at (5.5,0.5) [] {\small $\mathcal{E}_6^{(5)}$};
\node at (6.5,0.5) [] {\small $\mathcal{E}_6^{(6)}$};
\node at (7.5,0.5) [] {\small $\mathcal{E}_6^{(7)}$};
\node at (8.5,0.5) [] {\small $\mathcal{E}_6^{(8)}$};
\node at (9.5,0.5) [] {\small $\mathcal{E}_6^{(9)}$};

\node at (0.5,-0.5) [] {\small 25};
\node at (1.5,-0.5) [] {\small 33};
\node at (2.5,-0.5) [] {\small 77};
\node at (3.5,-0.5) [] {\small 109};
\node at (4.5,-0.5) [] {\small 175};
\node at (5.5,-0.5) [] {\small 266};
\node at (6.5,-0.5) [] {\small 341};
\node at (7.5,-0.5) [] {\small 871};
\node at (8.5,-0.5) [] {\small 1420};
\node at (9.5,-0.5) [] {\small 2750};

\node at (0.5,-1.5) [] {\small 25};
\node at (1.5,-1.5) [] {\small 33};
\node at (2.5,-1.5) [] {\small 77};
\node at (3.5,-1.5) [] {\small 109};
\node at (4.5,-1.5) [] {\small 175};
\node at (5.5,-1.5) [] {\small 260};
\node at (6.5,-1.5) [] {\small 326};
\node at (7.5,-1.5) [] {\small 781};
\node at (8.5,-1.5) [] {\small 1033};
\node at (9.5,-1.5) [] {\small 0};

\node at (0.5,-2.5) [] {\small};
\node at (1.5,-2.5) [] {\small 2.5s};
\node at (2.5,-2.5) [] {\small 2.6s};
\node at (3.5,-2.5) [] {\small 2.6s};
\node at (4.5,-2.5) [] {\small 2.7s};
\node at (5.5,-2.5) [] {\small 2.8s};
\node at (6.5,-2.5) [] {\small 2.9s};
\node at (7.5,-2.5) [] {\small 3.3s};
\node at (8.5,-2.5) [] {\small 4.3s};
\node at (9.5,-2.5) [] {\small 6.6s};

\end{tikzpicture}
}

\subfloat[][The filtration \eqref{eq:filtrationBar} in the case $n=6$. The vectors $\{\delta_S \notin \mathscr{B}_6\}$ have been added in the order $\{\delta_{1, 2, 3}, \delta_{1, 2}, \delta_{1, 2, 4}, \delta_{1, 3}, \delta_{1, 3, 5}, \delta_{2, 4}, \delta_{3, 5}, \delta_{4, 6}, \delta_{5, 6}\}$.]
{\label{tab:filtrationBar}
\begin{tikzpicture}[scale=0.875]
\draw (-4,0) -- (-4,-3);
\draw (0,1) -- (0,-3);
\draw (1,1) -- (1,-3);
\draw (2,1) -- (2,-3);
\draw (3,1) -- (3,-3);
\draw (4,1) -- (4,-3);
\draw (5,1) -- (5,-3);
\draw (6,1) -- (6,-3);
\draw (7,1) -- (7,-3);
\draw (8,1) -- (8,-3);
\draw (9,1) -- (9,-3);
\draw (10,1) -- (10,-3);

\draw (0,1) -- (10,1);
\draw (-4,0) -- (10,0);
\draw (-4,-1) -- (10,-1);
\draw (-4,-2) -- (10,-2);
\draw (-4,-3) -- (10,-3);

\node at (-2,-0.5) [] {\small number of facets};
\node at (-2,-1.5) [] {\small $\Gamma\big(\overline{\mathcal{F}_6},\overline{\mathcal{E}_6}^{(i)}\big)$};
\node at (-2,-2.5) [] {\small computing time\footnotemark[3]};

\node at (0.5,0.5) [] {\small $\overline{\mathcal{E}_6}^{(0)}$};
\node at (1.5,0.5) [] {\small $\overline{\mathcal{E}_6}^{(1)}$};
\node at (2.5,0.5) [] {\small $\overline{\mathcal{E}_6}^{(2)}$};
\node at (3.5,0.5) [] {\small $\overline{\mathcal{E}_6}^{(3)}$};
\node at (4.5,0.5) [] {\small $\overline{\mathcal{E}_6}^{(4)}$};
\node at (5.5,0.5) [] {\small $\overline{\mathcal{E}_6}^{(5)}$};
\node at (6.5,0.5) [] {\small $\overline{\mathcal{E}_6}^{(6)}$};
\node at (7.5,0.5) [] {\small $\overline{\mathcal{E}_6}^{(7)}$};
\node at (8.5,0.5) [] {\small $\overline{\mathcal{E}_6}^{(8)}$};
\node at (9.5,0.5) [] {\small $\overline{\mathcal{E}_6}^{(9)}$};

\node at (0.5,-0.5) [] {\small 16};
\node at (1.5,-0.5) [] {\small 25};
\node at (2.5,-0.5) [] {\small 34};
\node at (3.5,-0.5) [] {\small 49};
\node at (4.5,-0.5) [] {\small 108};
\node at (5.5,-0.5) [] {\small 239};
\node at (6.5,-0.5) [] {\small 491};
\node at (7.5,-0.5) [] {\small 869};
\node at (8.5,-0.5) [] {\small 1419};
\node at (9.5,-0.5) [] {\small 2750};

\node at (0.5,-1.5) [] {\small 1};
\node at (1.5,-1.5) [] {\small 0};
\node at (2.5,-1.5) [] {\small 0};
\node at (3.5,-1.5) [] {\small 0};
\node at (4.5,-1.5) [] {\small 0};
\node at (5.5,-1.5) [] {\small 0};
\node at (6.5,-1.5) [] {\small 0};
\node at (7.5,-1.5) [] {\small 0};
\node at (8.5,-1.5) [] {\small 0};
\node at (9.5,-1.5) [] {\small 0};

\node at (0.5,-2.5) [] {\small};
\node at (1.5,-2.5) [] {\small 2.5s};
\node at (2.5,-2.5) [] {\small 2.6s};
\node at (3.5,-2.5) [] {\small 2.6s};
\node at (4.5,-2.5) [] {\small 2.7s};
\node at (5.5,-2.5) [] {\small 2.7s};
\node at (6.5,-2.5) [] {\small 2.8s};
\node at (7.5,-2.5) [] {\small 3.3s};
\node at (8.5,-2.5) [] {\small 4.1s};
\node at (9.5,-2.5) [] {\small 5.9s};

\end{tikzpicture}
}
\caption{\label{tab:incremental_N6} The filtrations of cones $\mathcal{E}_6^{(i)}$ in $\boldsymbol{W}_6$ and $\overline{\mathcal{E}_6}^{(i)}$ in $\boldsymbol{U} = \QQ\langle\mathscr{B}_6\rangle$.}
\end{center}
\end{table}




\begin{thebibliography}{10}

\bibitem{lpSolve}
M.~Berkelaar, K.~Eikland, and P.~Notebaert.
\newblock {\tt lpSolve v.~5.5}.
\newblock Available at \url{https://cran.r-project.org/web/packages/lpSolve/}.

\bibitem{FarGib:03}
G.~Farkas and A.~Gibney.
\newblock The {M}ori cones of moduli spaces of pointed curves of small genus.
\newblock {\em Trans.~Amer.~Math.~Soc.}, 355(3):1183--1199, 2003.

\bibitem{Fontanari:05}
C.~Fontanari.
\newblock A remark on the ample cone of {$\overline{\mathscr M}_{g,n}$}.
\newblock {\em Rend.~Sem.~Mat.~Univ.~Politec.~Torino}, 63(1):9--14, 2005.

\bibitem{polymake}
E.~Gawrilow and M.~Joswig.
\newblock {\tt polymake}: a framework for analyzing convex polytopes.
\newblock In G.~Kalai and G.~M. Ziegler, editors, {\em Polytopes ---
  Combinatorics and Computation}, pages 43--74. Birkh\"auser, 2000.
\newblock Available at \url{https://polymake.org/}.

\bibitem{GKM:02}
A.~Gibney, S.~Keel, and I.~Morrison.
\newblock Towards the ample cone of {$\overline M_{g,n}$}.
\newblock {\em J.~Amer.~Math.~Soc.}, 15(2):273--294, 2002.

\bibitem{M2}
D.~R. Grayson and M.~E. Stillman.
\newblock {\it Macaulay2}, a software system for research in algebraic
  geometry.
\newblock Available at \url{http://www.math.uiuc.edu/Macaulay2/}.

\bibitem{Keel:92}
S.~Keel.
\newblock Intersection theory of moduli space of stable {$n$}-pointed curves of
  genus zero.
\newblock {\em Trans.~Amer.~Math.~Soc.}, 330(2):545--574, 1992.

\bibitem{KMK:96}
S.~Keel and J.~McKernan.
\newblock Contractible extremal rays on $\overline{M}_{0,n}$.
\newblock {\em arXiv e-prints}, 1996.
\newblock Available at \href{http://arxiv.org/abs/alg-geom/9607009}{\tt
  arXiv:alg-geom/9607009}.

\bibitem{Larsen:12}
P.~L. Larsen.
\newblock Fulton's conjecture for {$\overline M_{0,7}$}.
\newblock {\em J.~Lond.~Math.~Soc.~(2)}, 85(1):1--21, 2012.

\bibitem{Pixton:13}
A.~Pixton.
\newblock A nonboundary nef divisor on {$\overline M_{0,12}$}.
\newblock {\em Geom.~Topol.}, 17(3):1317--1324, 2013.

\end{thebibliography}

\vspace{0.5cm}

\noindent
Claudio Fontanari, Riccardo Ghiloni, Paolo Lella \newline
Dipartimento di Matematica, Universit\`a degli Studi di Trento \newline 
Via Sommarive 14, 38123 Povo Trento (Italy) \newline
{\tt\{\href{mailto:claudio.fontanari@unitn.it}{\tt claudio.fontanari}, \href{mailto:riccardo.ghiloni@unitn.it}{\tt riccardo.ghiloni}, \href{mailto:paolo.lella@unitn.it}{\tt paolo.lella}\}@unitn.it}

\end{document}